\documentclass[11pt]{article}

\usepackage[T1]{fontenc}
\usepackage{lmodern}

\usepackage{amsmath}
\usepackage{amssymb}
\usepackage{amscd}
\usepackage{amsthm}
\usepackage{indentfirst}

\newcommand{\Z}{\ensuremath{\mathbb{Z}}}
\newcommand{\Q}{\ensuremath{\mathbb{Q}}}

\input cyracc.def \font\tencyr=wncyr10 \def\russe{\tencyr\cyracc} 
\def\Sha{\text{\russe{Sh}}}

\usepackage[margin=2.5cm]{geometry}

\theoremstyle{plain}
\newtheorem{theorem}{Theorem}[section]
\newtheorem{lemma}[theorem]{Lemma}
\newtheorem{corollary}[theorem]{Corollary}

\theoremstyle{definition}
\newtheorem{definition}[theorem]{Definition}
\newtheorem{remark}[theorem]{Remark}

\DeclareMathOperator{\Gal}{Gal}
\DeclareMathOperator{\GL}{GL}

\newcommand{\1}{\mathbf{1}}
\newcommand{\A}{\mathbf{A}}
\newcommand{\E}{\mathcal{E}}
\newcommand{\cyc}{\mathrm{cyc}}

\DeclareMathOperator{\im}{im}

\begin{document}

\title{Selmer groups are intersection of two direct summands of the adelic cohomology}

\author{Florence Gillibert, Jean Gillibert, Pierre Gillibert, Gabriele Ranieri}

\date{February 2019}

\maketitle

\begin{abstract}
We give a positive answer to a conjecture by Manjul Bhargava, Daniel M. Kane, Hendrik W. Lenstra Jr., Bjorn Poonen and Eric Rains, concerning the cohomology of torsion subgroups of elliptic curves over global fields.
This implies that, given a global field $k$ and an integer $n$, for $100\%$ of elliptic curves $E$ defined over $k$, the $n$-th Selmer group of $E$ is the intersection of two direct summands of the adelic cohomology group $H^1(\A,E[n])$.
We also give examples of elliptic curves for which the conclusion of this conjecture does not hold.
\end{abstract}



\section{Introduction}


Let $k$ be a global field (a number field or the function field of a curve over a finite field), and let $\Omega$ be the set of places of $k$. If $v$ is a place of $k$, we denote by $k_v$ the localisation of $k$ at $v$, and by $\mathcal{O}_v$ its ring of integers. We denote by $\A$ the adele ring of $k$, defined as the restricted product
$$
\A:=\prod_{v\in\Omega}' (k_v,\mathcal{O}_v).
$$

In this paper, all cohomology groups are computed with respect to the fppf topology. When $M$ is a smooth commutative $k$-group scheme, the fppf cohomology groups $H^i(k,M)$ agree with the Galois cohomology groups.

If $M$ is a $k$-group scheme, we define $\Sha^i(k,M)$ for $i=1,2$ as
$$
\Sha^i(k,M):=\ker\bigg(H^i(k,M)\longrightarrow \prod_{v\in\Omega} H^i(k_v,M)\bigg).
$$

If $E$ is an abelian variety defined over $k$, we let $E(\A):=\prod_{v\in\Omega} E(k_v)$. Furthermore, if $n>0$ is an integer, we let
$$
H^1(\A,E[n]):= \prod_{v\in\Omega}' (H^1(k_v,E[n]),H^1(\mathcal{O}_v,\E_v[n])),
$$
where $\E_v$ denotes the N{\'e}ron model of $E$ over $\mathcal{O}_v$. This definition seems a bit \emph{ad hoc}. Nevertheless, it was proved by \v{C}esnavi\v{c}ius \cite{ces15} that this group \emph{is} the first fppf cohomology group of $E[n]$ over $\A$; see \cite[\S{}6.2]{BKLPR15} for further details.

We have a natural commutative diagram
$$
\begin{CD}
E(k)/n @>\delta_n>> H^1(k,E[n]) \\
@VVV @VV\beta_n V \\
E(\A)/n @>\Delta_n>> H^1(\A,E[n]) \\
\end{CD}
$$
in which the horizontal maps are the Kummer maps, and the vertical maps are obtained by localisation.

It was proved in \cite[Proposition~6.7]{BKLPR15} (see also \cite{G-G}) that the Kummer exact sequence splits. This implies that the image of the Kummer map $\Delta_n$ is a direct summand of $H^1(\A,E[n])$. It was conjectured in \cite[Conjecture 6.9]{BKLPR15} that the image of $\beta_n$ should also be a direct summand of $H^1(\A,E[n])$ for $100\%$ of elliptic curves $E$ defined over $k$ (for the exact meaning of this $100\%$, see the comments after Corollary~\ref{C1}). The motivation for this conjecture is the following: if $\Sha^1(k,E[n])=0$ (which holds for $100\%$ of elliptic curves over $k$), then  $\beta_n$ is injective (by the very definition of $\Sha^1$), hence the $n$-th Selmer group of $E$ is isomorphic to the intersection of the images of $\Delta_n$ and of $\beta_n$ in $H^1(\A,E[n])$.

Based on some heuristic arguments---and by analogy with the case when $n$ is a prime number---Selmer groups of elliptic curves are modelled in \cite{BKLPR15} as the intersection of two \emph{direct summands} of $H^1(\A,E[n])$. See \cite[\S{}1.6]{BKLPR15} for further details.

The aim of this short note is to prove the following result about the map $\beta_n$.

\begin{theorem}
\label{T:Main}
Let $E$ be an abelian variety defined over $k$, and let $E^t$ be its dual abelian variety. Let $n>1$ be an integer such that $\Sha^1(k,E[m])=\Sha^1(k,E^t[m])=0$ for every $m\mid n$ (which implies in particular that $\beta_n$ is injective). Then the image of $\beta_n$ is a direct summand of $H^1(\A,E[n])$.
\end{theorem}

We deduce the following corollary, which positively answers to \cite[Conjecture 6.9]{BKLPR15}, hence completes the arithmetic justification for the modelling of the Selmer group given in \cite{BKLPR15}.

\begin{corollary}
\label{C1}
Let $k$ be a global field and let $n>1$ be an integer. For $100\%$ of the elliptic curves $E$ defined over $k$, the map $\beta_n$ is injective, and its image is a direct summand of $H^1(\A,E[n])$.
\end{corollary}

Here, the set of elliptic curves over $k$ is ordered by height, as detailed in \cite[Conjecture~1.1]{PR12}. We say that some property $\mathbf{P}$ holds for $100\%$ of elliptic curves over $k$ if the set of elliptic curves over $k$ satisfying $\mathbf{P}$ has natural density $1$ with respect to the ordering of elliptic curves by height. For example, it is a theorem that $E(\Q)_\mathrm{tors}=0$ holds for $100\%$ of elliptic curves $E$ over $\Q$.

In the case when $k=\Q$, we obtain, as a consequence of work by Paladino, Ranieri and Viada \cite{PRV}, the following stronger result.

\begin{corollary}
\label{C2}
Let $n > 1$ be an integer coprime to $6$. Then for all elliptic curve $E$ defined over~$\Q$, the map $\beta_n$ is injective, and its image is a direct summand of $H^1(\A,E[n])$.
\end{corollary}

Given a prime $p$, one can also consider the map $\beta_{p^\infty}:H^1(k,E[p^\infty])\to H^1(\A,E[p^\infty])$ which is the direct limit of the $\beta_{p^r}$. We discuss the properties of this map in Remark~\ref{R:limit}.

In another direction, it is natural to ask if the hypotheses of Theorem~\ref{T:Main} are required for the conclusion to hold. The answer is positive in a strong sense.

\begin{theorem}
\label{T:contreexemple}
Let $E$ be an elliptic curve over a number field $k$, without complex multiplication, and let $p\equiv 2\pmod{3}$ be a large enough prime number. Then there exists a number field $L/k$ such that the map
$$
\beta_{p^3}: H^1(L,E[p^3]) \to H^1(\A_L,E[p^3])
$$
is not injective, and its image is not a direct summand of the target group.
\end{theorem}


\subsection*{Acknowledgements}

The first author was financed by FONDECYT REGULAR no.~1180489. The second author was supported by the CIMI Excellence program while visiting the \emph{Centro di Ricerca Matematica Ennio De Giorgi} during the autumn of 2017. The third author was supported by project no.~P27600 of the Austrian Science Fund (FWF).
Part of this research was conducted while the authors were enjoying the hospitality of the Institut de Math{\'e}matiques de Vimont (IMV). We wish to thank the IMV and its staff for its generous support and nice working atmosphere.
We also thank Roberto Dvornicich and Christian Wuthrich for their helpful remarks, and Philippe Satg{\'e} for reading a preliminary version of this note.


\section{Direct summands}


\subsection{Divisibility preserving maps}

If $M$ is an abelian group (or, more generally, an object in some abelian category), and if $m>0$ is an integer, we denote by $[m]:M\to M$ the multiplication by $m$ map on $M$, by $M[m]$ its kernel, by $mM$ its image, and by $M/m$ its cokernel. We say that $M$ is \emph{$m$-torsion} if $mM=0$.

\begin{definition}
\begin{enumerate}
\item[(1)] Let $A$ be an abelian group, $a\in A$, and $m\ge 1$ be an integer. We say that $a$ is \emph{divisible by $m$ in $A$} if there exists $a'$ in $A$ such that $a=ma'$. 
\item[(2)] Let $f:A\to B$ be a morphism of abelian groups. We say that $f$ \emph{preserves divisibility} if for every integer $m\ge 1$ and for every $a\in A$, if $f(a)$ is divisible by $m$ in $B$, then $a$ is divisible by $m$ in $A$. Equivalently, for every $m\ge 1$, the induced morphism $A/m\to B/m$ is injective.
\item[(3)] Let $n\ge 1$ be an integer. We say that $f$ \emph{preserves $n$-divisibility} if for every integer $m\mid n$ and for every $a\in A$, if $f(a)$ is divisible by $m$ in $B$, then $a$ is divisible by $m$ in $A$. Equivalently, for every $m\mid n$, the induced morphism $A/m\to B/m$ is injective.
\end{enumerate}
\end{definition}

\begin{lemma}\label{L:EquivDivisibilite}
Consider an embedding $f:A\to B$ of $n$-torsion abelian groups. Then the following statements are equivalent:
\begin{enumerate}
\item[(1)] $f(A)$ is a direct summand of $B$;
\item[(2)] $f$ preserves divisibility;
\item[(3)] $f$ preserves $n$-divisibility.
\end{enumerate}
\end{lemma}

\begin{proof}
Assume $(1)$. Let $C$ be a subgroup of $B$ such that $B=f(A)\bigoplus C$. Let $m\ge 1$ be an integer and let $a\in A$ be such that $f(a)$ is divisible by $m$. Let  $b\in B$ be such that $f(a)=mb$; let us write $b=f(x)+y$ where $x\in A$, and $y\in C$. Then $f(a) = m b = mf(x)+my$, so $my=f(a-mx)$, hence $my\in f(A)\cap C$, thus $f(a-mx)=my=0$. As $f$ is an embedding it follows that $a=mx$. Therefore $f$ preserves divisibility. Thus $(1)\Rightarrow (2)$.

Conversely, if $(2)$ holds then according to \cite[Chap.~V,~Thm.~29.1]{fuchs70}, $f(A)$ is a pure subgroup of $B$. Moreover, the quotient $B/f(A)$ is $n$-torsion, hence is a direct sum of cyclic groups, according to the first Pr{\"u}fer Theorem (see \cite[Chap.~III,~\S{}17]{fuchs70}). It follows from \cite[Chap.~V,~Thm.~28.2]{fuchs70} that $f(A)$ is a direct summand of $B$. Hence $(2)\Rightarrow (1)$.

Finally, $(2)\Leftrightarrow (3)$ is immediate, given the fact that $A$ and $B$ are $n$-torsion groups.
\end{proof}

\begin{lemma}\label{L:DiagrammeImpliqueDivisibilite}
Let $n\ge 1$. Consider a commutative diagram of abelian groups
\[
\begin{CD}
A @>f>> B @>g>> C \\
@. @V\beta VV  @V\gamma VV\\
@. B' @>g'>> C'
\end{CD}
\]
Assume that the following statements hold:
\begin{enumerate}
\item[(1)] $B$ and $B'$ are $n$-torsion groups;
\item[(2)] $\beta$ is an embedding;
\item[(3)] $\gamma$ preserves $n$-divisibility;
\item[(4)] $C[n]\subseteq \im g$;
\item[(5)] $\ker g \subseteq \im f$;
\item[(6)] $\beta\circ f$ preserves $n$-divisibility.
\end{enumerate}
Then $\beta(B)$ is a direct summand of $B'$.
\end{lemma}

\begin{proof}
Let $m$ be an integer dividing $n$. Let $b\in B$ be such that $\beta(b)$ is divisible by $m$ in $B'$. Hence $\beta(b)=mt$ for some $t\in B'$. Thus $\frac{n}{m}\beta(b) = \frac{n}{m} mt=nt=0$, because $B'$ is $n$-torsion. The map $\beta$ being an embedding, it follows that $\frac{n}{m}b=0$.

Note that $\gamma(g(b))=g'(\beta(b))$ is divisible by $m$. As $\gamma$ preserves $n$-divisibility, there exists $c\in C$ such that $mc=g(b)$. Then we have
$nc=\frac{n}{m}mc=\frac{n}{m}g(b)=g(\frac{n}{m}b)=0$, hence $c$ belongs to $C[n]$. But $C[n]\subseteq \im g$, so there exists $s\in B$ such that $g(s)=c$.

Let $y:=b-ms$. Note that $g(y) = g(b) - mg(s) = g(b)-mc=0$. As $\ker g \subseteq \im f$, there exists $x\in A$ such that $f(x)=y$. Hence $\beta(f(x)) = \beta(y) = \beta(b) - m\beta(s)=m(t-\beta(s))$ is divisible by $m$ in $B'$. But $\beta\circ f$ preserves $n$-divisibility, so $x$ is divisible by $m$ in $A$. Let $a\in A$ be such that $x=ma$, then $m(f(a)+s) = f(ma) + ms = f(x) + ms = y + ms = b$, hence $b$ is divisible by $m$ in $B$.

The reasoning above proves that $\beta$ preserves $n$-divisibility; this map being an embedding of $n$-torsion abelian groups, it follows from Lemma~\ref{L:EquivDivisibilite} that $\beta(B)$ is a direct summand of $B'$.
\end{proof}


\subsection{Abelian varieties}


From now on, the notations and hypotheses from the introduction are in use. In particular, $k$ denotes a global field.

\begin{lemma}\label{L:alphaPreserveDivisibilite}
Let $E$ be an abelian variety defined over $k$. Let $n>1$ be an integer such that $\Sha^1(k,E[m])=0$ for every $m\mid n$. Then the map
$$
E(k)/n \longrightarrow \prod_{v\in \Omega} E(k_v)/n
$$
preserves $n$-divisibility.
\end{lemma}

\begin{proof}
Let $m$ be an integer dividing $n$. Consider the following commutative diagram
\[
\begin{CD}
E(k)/m @>\delta_m>>  H^1(k,E[m])\\
@V\alpha VV  @V\beta VV  \\
\prod_{v\in \Omega} E(k_v)/m  @>>> \prod_{v\in \Omega} H^1(k_v,E[m]).
\end{CD}
\]

As $\Sha^1(k,E[m])=0$, the morphism $\beta$ is an embedding. The horizontal maps being embeddings, it follows that $\alpha$ is an embedding, hence the result.
\end{proof}

\begin{lemma}\label{L:gammaPreserveDivisibilite}
Let $E$ be an abelian variety defined over $k$, and let $E^t$ be its dual abelian variety. Let $n>1$ be an integer such that $\Sha^1(k,E^t[m])=0$ for every $m\mid n$. Then the map
\begin{equation}
\label{eq:gamma}
H^1(k,E) \longrightarrow \prod_{v\in \Omega} H^1(k_v,E)
\end{equation}
preserves $n$-divisibility.
\end{lemma} 

\begin{proof}
This follows from \cite[Lemma~5]{CV17}, but for the sake of completeness we give a short proof. By an elementary cohomological argument, if $\Sha^2(k,E[m])=0$ then the map
$$
H^1(k,E)/m \longrightarrow \prod_{v\in \Omega} H^1(k_v,E)/m
$$
is injective. Therefore, if $\Sha^2(k,E[m])=0$ for every $m\mid n$, then by definition the map \eqref{eq:gamma} preserves $n$-divisibility. The result follows from Tate's duality results (see \cite[Theorem~3.1]{tate63} in the case when $m$ is invertible in $k$, and \cite[Theorem~1.1]{GA09} in the case when $m$ is a power of the caracteristic of $k$): for every $m>0$, there is a perfect pairing of finite groups
$$
\Sha^1(k,E^t[m])\times \Sha^2(k,E[m])\longrightarrow \Q/\Z.
$$
Hence the triviality of $\Sha^2(k,E[m])$ is equivalent to that of $\Sha^1(k,E^t[m])$.
\end{proof}

\begin{lemma}\label{L:fPreserveDivisibilite}
Let $E$ be an abelian variety defined over $k$, and let $n>1$ be an integer. Then the Kummer map $\delta_n:E(k)/n\to  H^1(k,E[n])$ preserves $n$-divisibility.
\end{lemma}

\begin{proof}
From \cite[Proposition~6.7]{BKLPR15} the image of $\delta_n:E(k)/n\to  H^1(k,E[n])$ is a direct summand of $H^1(k,E[n])$, see also \cite[Corollary 1.2]{G-G}. Lemma~\ref{L:EquivDivisibilite} implies that $\delta_n$ preserves $n$-divisibility.
\end{proof}

\begin{theorem}\label{T:PreserveDivisibiliteLocalH1}
Let $E$ be an abelian variety defined over $k$, and let $E^t$ be its dual abelian variety. Let $n>1$ be an integer such that $\Sha^1(k,E[m])=\Sha^1(k,E^t[m])=0$ for every $m\mid n$. Then the map
$$
H^1(k,E[n]) \longrightarrow \prod_{v\in \Omega} H^1(k_v,E[n])
$$
is injective, and its image is a direct summand of the target group.
\end{theorem}

\begin{proof}
Consider the commutative diagram
\[
\begin{CD}
E(k)/n @>f>>  H^1(k,E[n]) @>g>> H^1(k,E)\\
@V\alpha VV  @V\beta VV @V\gamma VV  \\
\prod_{v\in \Omega} E(k_v)/n  @>f'>> \prod_{v\in \Omega} H^1(k_v,E[n]) ) @>g'>> \prod_{v\in \Omega} H^1(k_v,E)
\end{CD}
\]
in which $f$ and $f'$ are the Kummer maps. We note that the first line fits into the Kummer exact sequence
$$
\begin{CD}
0 @>>> E(k)/n @>>>  H^1(k,E[n]) @>>> H^1(k,E)[n] @>>> 0. \\
\end{CD}
$$
In particular, $\ker g=\im f$ and $\im g=H^1(k,E)[n]$.

The kernel of $\beta$ is $\Sha^1(k,E[n])=0$, hence $\beta$ is an embedding. Lemma~\ref{L:gammaPreserveDivisibilite} implies that $\gamma$ preserves $n$-divisibility. By Lemma~\ref{L:alphaPreserveDivisibilite}, the map $\alpha$ preserves $n$-divisibility, and by Lemma~\ref{L:fPreserveDivisibilite}, the map $f'$ preserves $n$-divisibility. It follows, by composition of $n$-divisibility preserving maps, that $f'\circ \alpha$ preserves $n$-divisibility, hence by commutativity of the diagram above $\beta\circ f$ preserves $n$-divisibility. All conditions of Lemma~\ref{L:DiagrammeImpliqueDivisibilite} being satisfied, one concludes that the image of $\beta$ is a direct summand of the target group.
\end{proof}


\subsection{Proofs of the main results}


\begin{proof}[Proof of Theorem~\ref{T:Main}]
If $A\leq B\leq B_0$ are abelian groups such that $A$ is a direct summand of $B_0$, then $A$ is a direct summand of $B$: indeed, if $B_0=A\oplus C$, then  $B=A\oplus (C\cap B)$. Given the fact that $H^1(\A,E[n])$ is a subgroup of $\prod_{v\in \Omega} H^1(k_v,E[n])$, the result follows from Theorem~\ref{T:PreserveDivisibiliteLocalH1}.
\end{proof}

\begin{proof}[Proof of Corollary~\ref{C1}]
Let $k$ be a global field, and let $n>0$ be an integer. It follows from Proposition~6.1 and Remark~6.2 of \cite{BKLPR15} that the property
\begin{center}
$\Sha^1(k,E[m])=0$ for every $m\mid n$
\end{center}
holds for $100\%$ of elliptic curves $E$ defined over $k$.
But $E^t[m]\simeq E[m]$ by self-duality of elliptic curves, hence $\Sha^1(k,E^t[m])=0$ also holds for every $m\mid n$. The result then follows from Theorem~\ref{T:Main}.
\end{proof}

\begin{proof}[Proof of Corollary~\ref{C2}]
Let $E$ be an elliptic curve defined over $\Q$. According to \cite{PRV}, for every prime $p\geq 5$ and for every integer $r\geq 1$, $\Sha^1(\Q,E[p^r])=0$. The result follows from Theorem~\ref{T:Main} 
\end{proof}

\begin{remark}
\label{R:limit}
Let $k$ be a global field, let $p$ be a prime number, and let $E$ be an elliptic curve defined over $k$ such that $\Sha^1(k,E[p^r])=0$ for every $r>0$ (given $k$ and $p$, this holds for $100\%$ of $E$).
It follows from Theorem~\ref{T:Main} that, for each $r>0$, $\beta_{p^r}$ is injective and preserves divisibility. By passing to the direct limit when $r\to+\infty$, we get a map
$$
\beta_{p^\infty}:H^1(k,E[p^\infty])\longrightarrow H^1(\A,E[p^\infty])
$$
which is injective and preserves divisibility (in other terms, its image is a \emph{pure subgroup} of the target group). Nevertheless, we do not expect the image of $\beta_{p^\infty}$ to be a direct summand of the target group, because the retractions of the various maps $\beta_{p^r}$ have no reason to be compatible with each other.
\end{remark}


\section{Non-direct summands}


The aim of this section is to construct elliptic curves over number fields for which the conclusion of Theorem~\ref{T:Main} does not hold.

The following lemma allows one to obtain non injective morphisms whose image is not a direct summand of the target group. 

\begin{lemma}
\label{L:pasinjectif}
Consider a morphism $f:A\to B$ of $n$-torsion abelian groups. Suppose that there exists $m\mid n$ and $P\in A$ such that:
\begin{enumerate}
\item[(i)] $P$ is not $m$-divisible but $f(P)$ is $m$-divisible;
\item[(ii)] any element of $\ker(f)$ is $m$-divisible in $A$.
\end{enumerate}
Then $f(A)$ is not a direct summand of $B$. 
\end{lemma}

\begin{proof}
Let $\overline{f}:A/\ker(f) \to B$ be the morphism induced by $f$. From (i) and (ii) one deduces that the image $\overline{P}$ of $P$ in $A/\ker(f)$ is not $m$-divisible in $A/\ker(f)$. On the other hand, $\overline{f}(\overline{P})=f(P)$ is $m$-divisible by hypothesis (i). As $\overline{f}$ is an embedding and does not preserve divisibility, by Lemma~\ref{L:EquivDivisibilite} we deduce that $f(A)=\overline{f}(A/\ker(f))$ is not a direct factor of $B$. 
\end{proof}

Following the tradition, if some finite group $G$ acts on some abelian group $M$, we define
$$
H^1_\cyc(G,M):= \ker\bigg(H^1(G,M)\longrightarrow \prod_{\substack{\Gamma\leq G \\ \Gamma \text{ cyclic}}} H^1(\Gamma,M)\bigg).
$$
(In some papers, this group is denoted by $H^1_{\mathrm{loc}}(G,M)$).

Let $k$ be a number field, and let $M$ be a finite $k$-group scheme (equivalently, a finite $\Gal(\bar{k}/k)$-module). Let us denote by $k(M)$ the smallest extension of $k$ over which $M$ becomes constant (equivalently, the Galois action becomes trivial over $k(M)$).

It follows from Chebotarev's density theorem (see \cite[Lemma~6.3]{BKLPR15}) that $\Sha^1(k,M)$ can be identified with a subgroup of
$$
H^1_\cyc(\Gal(k(M)/k),M).
$$

Moreover, if $k(M)/k$ is everywhere unramified, then all decomposition subgroups of $k(M)/k$ are cyclic, hence
$$
\Sha^1(k,M)=H^1_\cyc(\Gal(k(M)/k),M).
$$

The following lemma is a consequence of the construction of \cite[Section~5]{GR18}.

\begin{lemma}
\label{L:fg}
Let $p$ be a prime number such that $p\equiv 2\pmod{3}$.  Let
$$
H_2 = 
\bigg \{
\left(
\begin{array}{cc}
1 + p ( a - 2b ) & 3p ( b-a ) \\
-pb & 1 - p ( a - 2b ) \\
\end{array}
\right)  
\mid \ a, b \in \Z/ p^2 \Z
\bigg \}.
$$
Then $H_2$ is a subgroup of $\GL_2(\Z/p^2\Z)$. Let
$$
G_2 = \bigg\langle
\left( 
\begin{array}{cc} 
1 & -3 \\
1 & -2 \\
\end{array}
\right),
H_2
\bigg\rangle.
$$
Finally, let
$$
\pi:\GL_2(\Z/p^3\Z)\to \GL_2(\Z/p^2\Z)
$$
be the canonical surjection, and let  $G_3:=\pi^{-1}(G_2)$. Then
$$
H^1_\cyc(G_2,(\Z/p^2\Z)^2)=H^1_\cyc(G_3,(\Z/p^3\Z)^2)\neq 0.
$$
\end{lemma}

\begin{proof}
In \cite[Section~5]{GR18}, it was proved that
$$
H^1_\cyc(G_2,(\Z/p^2\Z)^2)\neq 0.
$$

Let $N:=\ker(\pi)$ be the subgroup of matrices congruent to identity modulo $p^2$. By construction, we have an exact sequence
$$
\begin{CD}
1 @>>> N @>>> G_3 @>>> G_2 @>>> 1.\\
\end{CD}
$$

Let $M_3:=(\Z/p^3\Z)^2$ and $M_2=M_3[p^2]$. Then $M_3^N=M_2$ and the $G_2$-module $M_2$ is isomorphic to the $G_2$-module $(\Z/p^2\Z)^2$ induced by the natural matrix action. The inflation-restriction exact sequence reads
\begin{equation}
\label{eq:infres}
\begin{CD}
0 @>>> H^1(G_2,M_2) @>\mathrm{inf}>> H^1(G_3,M_3) @>\mathrm{res}>> H^1(N,M_3)^{G_2}.
\end{CD}
\end{equation}

Let us prove that $H^1_\cyc(N,M_3)=0$. Let $Z:N\to M_3$ be a cocycle whose class belongs to $H^1_\cyc(N,M_3)$, then for each $h\in N$ there exists $m\in M_3$ such that $Z(h)=(h-\1)m$ (where $\1$ denotes the identity). Since $h-\1\equiv 0\pmod{p^2}$, we deduce that $Z(h)$ belongs to $p^2M_3$. Since $N$ acts trivially on $p^2M_3$, the cocycle $Z$ is in fact a morphism $N\to p^2 M_3$, and the result follows by the same argument as in \cite[Proposition~3.2]{DZ01}.

Since the restriction map sends $H^1_\cyc(G_3,M_3)$ to $H^1_\cyc(N,M_3)$, we deduce from the exact sequence \eqref{eq:infres} that $H^1_\cyc(G_3,M_3)$ belongs to the image of the inflation.

We claim that, for every cyclic subgroup $\Gamma\leq G_2$, $H^1(\Gamma,M_2)=0$. This implies that $H^1(G_2, M_2)=H^1_\cyc(G_2, M_2)$, hence one can deduce from the previous discussion, combined with the fact that the inflation sends $H^1_\cyc ( G_2, M_2 )$ to $H^1_\cyc ( G_3 , M_3 )$, that
$$
H^1_\cyc(G_2,(\Z/p^2\Z)^2)=H^1_\cyc(G_3,(\Z/p^3\Z)^2).
$$
 
Let
$$
g:=\left( 
\begin{array}{cc} 
1 & -3 \\
1 & -2 \\
\end{array}
\right)
$$
which has order $3$. Because $gH_2=H_2g$ (see \cite[Section~5]{GR18}), $H_2$ is a normal subgroup of $G_2$, and we have an exact sequence
$$
\begin{CD}
1 @>>> H_2 @>>> G_2 @>>> \langle g\rangle @>>> 1. \\
\end{CD}
$$

As $p\neq 3$, given a cyclic subgroup $\Gamma\leq G_2$, its $p$-Sylow subgroup is a subgroup of $H_2$, hence of the form $\langle h\rangle$ for some $h\in H_2$, and we have an embedding
$$
H^1(\Gamma,M_2) \hookrightarrow H^1(\langle h\rangle,M_2).
$$

We shall use the following classical fact: if $\gamma\in \GL_2(\Z/p^2\Z)$ is an element of order $r$, then the cohomology of the cyclic group $\langle\gamma \rangle$ acting on $M_2$ can be computed as
\begin{equation}
\label{cyclicH1}
H^1(\langle\gamma \rangle,M_2) = \ker(T_{\gamma})/(\gamma-\1)M_2
\end{equation}
where $T_{\gamma}$ denotes the matrix $\1+\gamma+\gamma^2+\cdots +\gamma^{r-1}$.

Let $h$ be any element of $H_2$ distinct from the identity. Then
$$
h - \1 = p\left(
\begin{array}{cc}
a - 2b & 3(b-a) \\
-b & -(a - 2b) \\
\end{array}
\right)=:pM_{a,b}
$$
One checks that $M_{a,b}$ is an invertible matrix (this calculation was made in details \cite[Section 5]{GR18}).
Then, since $M_{a,b}$ is invertible, the image of $h- \1$ is $pM_2=M_2[p]$. On the other hand, a simple calculation shows that
$$
T_h=\1 + h + \ldots + h^{p-1}=p\cdot\1
$$
whose kernel is $M_2[p]$. It follows from \eqref{cyclicH1} that $H^1(\langle h \rangle, M_2) = 0$, which concludes the proof.
\end{proof}


\begin{lemma}
\label{L:chbase1}
Let $E$ be an elliptic curve over a number field $k$, without complex multiplication, and let $p\equiv 2\pmod{3}$ be a large enough prime number. Let $K_2:=k(E[p^2])$ and $K_3:=k(E[p^3])$. Then there exists a number field $F/k$ with $E(F)[p]=0$, such that the morphism
$$
H^1_\cyc(\Gal(FK_2/F),E[p^2])\longrightarrow H^1_\cyc(\Gal(FK_3/F),E[p^3])
$$
is an isomorphism, and these groups are nonzero.
\end{lemma}

\begin{proof}
As $E$ has no complex multiplication, it follows from \cite{serre72} that, for $p$ large enough, the Galois representation $\Gal(\bar{k}/k)\to\GL_2(\Z_p)$ associated to the $p$-adic Tate module of $E$ is surjective. In particular, if $p$ is large enough, $\Gal(K_3/k)\simeq\GL_2(\Z/p^3\Z)$ and the natural action of $\Gal(K_3/k)$ on $E[p^3]$ can be identified with the action of $\GL_2(\Z/p^3\Z)$ on $M_3=(\Z/p^3\Z)^2$. Let $G_2$ and $G_3$ be the groups from Lemma~\ref{L:fg}, and let $F:=K_3^{G_3}$ be the field fixed by $G_3$. Then by elementary Galois theory $\Gal(FK_3/F)=G_3$ which acts on $E[p^3]$ as on $M_3$. It follows that $\Gal(FK_2/F)=G_2$.
Moreover, $G_3$ does not fix any point of $M_3$, hence $E(F)[p]=0$. The result follows from the conclusion of Lemma~\ref{L:fg}.
\end{proof}

\begin{lemma}
\label{L:chbase2}
Let $E$ be an elliptic curve over a number field $k$. Let $p$ be a prime number, let $n\geq 2$ be an integer, and let $K_n:=k(E[p^n])$. Then there exists a number field $L/k$ such that
\begin{enumerate}
\item[(a)] $L$ is linearly disjoint from $K_n$ over $k$;
\item[(b)] $LK_n/L$ is everywhere unramified;
\item[(c)] $H^1_\cyc(\Gal(LK_n/L),E[p^n])$ is contained in the image of the Kummer map $\delta_{p^n}:E(L)/p^n\to H^1(L,E[p^n])$.
\end{enumerate}
\end{lemma}

\begin{proof}
According to \cite[Corollary~1]{DZ07}, given $\xi\in H^1_\cyc(\Gal(K_n/k),E[p^n])$ there exists infinitely many extensions  $L/k$ such that
$( a )$ and $( b )$ hold, and $\xi$ belongs to the image of $\delta_{p^n}$. We note that $\Gal(LK_n/L)=\Gal(K_n/k)$, hence
$$
H^1_\cyc(\Gal(LK_n/L),E[p^n])\simeq H^1_\cyc(\Gal(K_n/k),E[p^n]).
$$
The result follows by finite induction, the group $H^1_\cyc(\Gal(K_n/k),E[p^n])$ being finite.
\end{proof}

\begin{proof}[Proof of Theorem~\ref{T:contreexemple}]
Let $K_2:=k(E[p^2])$ and $K_3:=k(E[p^3])$. According to Lemma~\ref{L:chbase1}, if $p$ is large enough, there exists some finite $F/k$ with $E(F)[p]=0$, such that
\begin{equation}
\label{eq:egalite23F}
H^1_\cyc(\Gal(FK_2/F),E[p^2])\simeq H^1_\cyc(\Gal(FK_3/F),E[p^3])\neq 0.
\end{equation}

According to Lemma~\ref{L:chbase2}, there exists $L/F$ such that
\begin{enumerate}
\item[(a)] $L$ is linearly disjoint from $FK_3$ over $F$;
\item[(b)] $LK_3/L$ is everywhere unramified;
\item[(c)] $H^1_\cyc(\Gal(LK_3/L),E[p^3])$ is contained in the image of the Kummer map $\delta_{p^3}:E(L)/p^3\to H^1(L,E[p^3])$.
\end{enumerate}

It follows from (a) that $\Gal(LK_3/L)=\Gal(FK_3/F)$ and similarly for $\Gal(LK_2/L)$, hence \eqref{eq:egalite23F} becomes
\begin{equation}
\label{eq:egalite23}
H^1_\cyc(\Gal(LK_2/L),E[p^2])\simeq H^1_\cyc(\Gal(LK_3/L),E[p^3])\neq 0.
\end{equation}

For the same reason, $E(L)[p]=E(F)[p]=0$.

Let us note that, since $LK_3/L$ is everywhere unramified by (b), 
\begin{equation}
\label{eq:H1cyc=sha}
\Sha^1(L,E[p^3])=H^1_\cyc(\Gal(LK_3/L),E[p^3]).
\end{equation}

The same holds for $\Sha^1(L,E[p^2])$ because $K_2L/L$ is a subextension of $K_3L/L$, hence is everywhere unramified. It follows from \eqref{eq:egalite23} that
\begin{equation}
\label{eq:sha2=sha3}
\Sha^1(L,E[p^2]) \simeq \Sha^1(L,E[p^3]).
\end{equation}

From \eqref{eq:egalite23} and \eqref{eq:H1cyc=sha} we deduce that $\Sha^1(L,E[p^3])\neq 0$, hence the map
$$
\beta_{p^3}: H^1(L,E[p^3]) \longrightarrow H^1(\A_L,E[p^3])
$$
is not injective. We shall prove that the hypotheses of Lemma~\ref{L:pasinjectif} hold for this map when taking $m=p^2$, which proves the result.

By Kummer theory, we have the following commutative diagram, with exact lines
$$
\begin{CD}
E(L)/p^2 @>\delta_{p^2}>> H^1(L,E[p^2])@>>> H^1(L,E)[p^2]\\
@VVV @VVV @VVV \\
E(L)/p^3 @>\delta_{p^3}>> H^1(L,E[p^3])@>>> H^1(L,E)[p^3].\\
\end{CD}
$$

From (c), we know that  $H^1_\cyc(\Gal(LK_3/L),E[p^3])$ is contained in the image of $\delta_{p^3}$, which is the same as the kernel of the lower right-hand side map.

On the other hand, the vertical map in the middle sends the subgroup $H^1_\cyc(\Gal(LK_2/L),E[p^2])$ to the subgroup $H^1_\cyc(\Gal(LK_3/L),E[p^3])$, and the vertical map on the right is injective (inclusion of subgroups). One deduces that $H^1_\cyc (\Gal(LK_2/L, E[p^2])$ is contained in the kernel of the upper right-hand side map, \emph{i.e.} the image of $\delta_{p^2}$. As
$$
\Sha^1(L,E[p^2])=H^1_\cyc(\Gal(LK_2/L),E[p^2])\neq 0
$$
we deduce that there exists a point $P\in E(L)$ which is everywhere locally $p^2$-divisible, but is not $p^2$-divisible. Then $\beta_{p^3}(\delta_{p^3}(P))$ is $p^2$-divisible in $H^1(\A_L,E[p^3])$ by construction. On the other hand, $\delta_{p^3}(P)$ is not $p^2$-divisible in $H^1(L,E[p^3])$, because $\delta_{p^3}$ preserves divisibility by Lemma~\ref{L:fPreserveDivisibilite}. Hence, the condition (i) from Lemma~\ref{L:pasinjectif} is satisfied.

It follows from the Mordell-Weil theorem and the fact that $E(L)$ has no $p$-torsion that
$$
E(L)/p^3\simeq (\Z/p^3\Z)^r
$$
where $r$ denotes the rank of $E$ over $L$.

We claim that the group $\Sha^1(L,E[p^2])$ is $p$-torsion. Indeed, any element of $\Sha^1(L,E[p^2])$ is of the form $\delta_{p^2}(P)$ for some $P\in E(L)$ which is locally $p^2$-divisible. Such a $P$ is locally $p$-divisible, hence is $p$-divisible, because $\Sha^1(L,E[p])=0$ (this results holds for any elliptic curve and any prime $p$, see \cite[Theorem~3.1]{DZ01}). Therefore, $P=pQ$ for some $Q\in E(L)$, hence $pP=p^2Q$ which is zero in $E(L)/p^2$, and the claim follows.

According to \eqref{eq:sha2=sha3}, it follows that any element of $\Sha^1(L,E[p^3])$ is $p$-torsion, hence is $p^2$-divisible in $E(L)/p^3 \simeq (\Z/p^3\Z)^r$. Consequently, any element of $\Sha^1(L,E[p^3])=\ker(\beta_{p^3})$ is $p^2$-divisible in $H^1(L,E[p^3])$, hence assumption (ii) of Lemma~\ref{L:pasinjectif} is satisfied.
\end{proof}

\begin{remark}
\begin{enumerate}
\item[a)] When going through our proof one can see that, in the setting of Theorem~\ref{T:contreexemple}, the natural map 
$$
E(L)/p^3 \to \prod_{v\in\Omega_L} E(L_v)/p^3
$$
is not injective, and its image is not a direct summand of the target group.
\item[b)] One can easily generalize the constructions of Theorem~\ref{T:contreexemple} by replacing $p^3$ by $p^r$ with $r\geq 3$. This yields examples for which $\beta_{p^r}$ is not injective and its image is not a direct summand of the target group.
\end{enumerate}
\end{remark}





\bigskip

Florence Gillibert and Gabriele Ranieri, Instituto de Matem{\'a}ticas, Pontificia Universidad Cat{\'o}lica de
Valpara{\'i}so, Blanco Viel 596, Cerro Bar{\'o}n, Valpara{\'i}so, Chile.

\emph{Email addresses:} \texttt{florence.gillibert@pucv.cl}, \texttt{gabriele.ranieri@pucv.cl}
\medskip

Jean Gillibert, Institut de Math{\'e}matiques de Toulouse, CNRS UMR 5219, 118, route de Narbonne, 31062 Toulouse Cedex 9, France.

\emph{Email address:} \texttt{jean.gillibert@math.univ-toulouse.fr}
\medskip

Pierre Gillibert, Institut f{\"u}r Diskrete Mathematik \& Geometrie, Technische Universit{\"a}t Wien, Wien,  {\"O}sterreich.

\emph{Email address:} \texttt{gillibert.pierre@tuwien.ac.at}, \texttt{pgillibert@yahoo.fr}


\end{document}